\documentclass[11pt]{amsart}

\usepackage[a4paper,margin=29mm]{geometry}
\usepackage{amsmath,amssymb,mathtools}
\usepackage{enumitem}
\usepackage{microtype}
\usepackage[colorlinks=true,linkcolor=blue,citecolor=blue,urlcolor=blue]{hyperref}
\usepackage[nameinlink,noabbrev]{cleveref}

\newcommand{\T}{\mathbb T}
\newcommand{\R}{\mathbb R}
\newcommand{\Z}{\mathbb Z}
\newcommand{\N}{\mathbb N}
\newcommand{\Q}{\mathbb Q}
\newcommand{\Diff}{\operatorname{Diff}}
\newcommand{\Homeo}{\operatorname{Homeo}}
\newcommand{\Leb}{\operatorname{Leb}}
\newcommand{\e}{\mathrm e}
\newcommand{\dd}{\,\mathrm d}
\newcommand{\cE}{\mathcal E}

\theoremstyle{plain}
\newtheorem{theorem}{Theorem}[section]
\newtheorem{proposition}[theorem]{Proposition}
\newtheorem{lemma}[theorem]{Lemma}

\theoremstyle{definition}

\theoremstyle{remark}

\title[Examples beyond bounded mean motion]
{Examples beyond Bounded Mean Motion\
for Quantitative Rigidity on the Two-Torus}

\author{Yinshan Chang}
\author{Jian Wang}
\author{Junchang Zhou}

\hypersetup{
  pdftitle={Examples beyond Bounded Mean Motion for Quantitative Rigidity on the Two-Torus},
  pdfauthor={Yinshan Chang, Jian Wang, and Junchang Zhou},
  pdfsubject={Semi-irrational and totally irrational area-preserving pseudo-rotations with quantitative deviation but without bounded mean motion}
}


\begin{document}

\begin{abstract}
This note supplies genuinely non-fibred examples for the manuscript
\emph{Rigidity on the Two-Torus and Sarnak's Conjecture}.  For every
$0<\delta<\tfrac12$, we construct $C^\infty$ Lebesgue-area-preserving
pseudo-rotations of $\T^2$ which satisfy the $(C,\delta)$-deviation
condition but do not have bounded mean motion.  We give both
semi-irrational and totally irrational rotation vectors and two
realizations: a controlled weakly mixing Anosov--Katok construction and an
explicit weakly mixing special-flow construction.  Weak mixing is used as
a conjugacy-invariant obstruction to every continuous circle-rotation
factor.  Consequently, none of the resulting maps is topologically
conjugate, by a linear or nonlinear change of coordinates, to a skew
product over a circle rotation.  In the special-flow realization, the
same lacunary Fourier series simultaneously gives weak mixing, the sharp
upper bound $O(n^\delta)$, and unbounded deviations; in fact no smaller
deviation exponent is possible.  The semi-irrational examples meet the
assumptions of Theorems~1 and~2 of the cited manuscript, whereas the
totally irrational examples meet those of Theorem~1.  Each construction
produces continuum many maps and continuum many topological conjugacy
classes of each rotation type.
\end{abstract}

\maketitle

\section{Purpose and main theorem}

Let $f\in\Homeo_0(\T^2)$ and let $F\colon\R^2\to\R^2$ be a lift.  If
the Misiurewicz--Ziemian rotation set \cite{MZ89} is a singleton, we call
$f$ a \emph{pseudo-rotation}.  Write the singleton as $\{\rho(F)\}$ and
set
\begin{equation}\label{eq:deviation-cocycle}
 D_n(F,z)=F^n(z)-z-n\rho(F).
\end{equation}
Bounded mean motion is the condition
\begin{equation}\label{eq:BMM}
 \sup_{n\geq1}\sup_{z\in\R^2}|D_n(F,z)|<\infty.
\end{equation}
The $(C,\delta)$-deviation condition of \cite{CWZ} is equivalent to
\begin{equation}\label{eq:Cdelta}
 \sup_{z\in\R^2}|D_n(F,z)|\leq Cn^\delta
 \qquad(n\geq1).
\end{equation}
Indeed, one takes $N=n$ in Definition~1.1 of \cite{CWZ}; conversely,
division by $n$ and $n^{\delta-1}\leq N^{\delta-1}$ for $n\geq N$
recover that definition.  Thus $\delta=0$ is bounded mean motion up to
the value of the constant, whereas every $\delta>0$ permits unbounded
sublinear displacement.

By a \emph{skew product over a circle rotation} we mean a map of the
form
\begin{equation}\label{eq:skew-product}
 S(x,y)=(x+\gamma,\Phi_x(y))
\end{equation}
on a circle bundle over $\T$, including the additive skew products
considered in \cite{CWZ}.  Saying that $f$ is not conjugate to a skew
product means that no torus homeomorphism, with arbitrary homotopy class
and no regularity or measure-preservation assumption, conjugates $f$ to
a map of the form \eqref{eq:skew-product}.

\begin{theorem}\label{thm:main}
Fix $0<\delta<\tfrac12$.  Both the Anosov--Katok method and the
special-flow method give a semi-irrational family and a totally
irrational family in $\Diff_0^\infty(\T^2)$ with the following
properties.
\begin{enumerate}[label=\textup{(\roman*)}]
 \item Every map preserves Lebesgue area and is weakly mixing with
       respect to Lebesgue measure.
 \item Its rotation set is a singleton.  In the semi-irrational family
       one can take the representative $(\alpha,0)$, where $\alpha$ is
       strongly super-Liouville and of strong non-Brjuno type.  In the
       totally irrational family the rotation vector can be chosen
       strongly super-Liouville.
 \item It satisfies \eqref{eq:Cdelta} for some $C<\infty$, but it does
       not have bounded mean motion.
 \item It is not topologically conjugate to a skew product over a circle
       rotation.  In particular, no nonlinear change of coordinates can
       restore a fibred structure.
\end{enumerate}
For the special-flow examples, the rotation vectors may be chosen
explicitly as $(\alpha,0)$ and $(\alpha^2,\alpha)$, and the exponent is
sharp: no $(C',\delta')$ condition holds for
$0\leq\delta'<\delta$.

Every semi-irrational example satisfies the assumptions of Theorem~1 of
\cite{CWZ} with H\"older exponent $a=1$ and of Theorem~2 for every finite
$k\geq2$.  Every totally irrational example satisfies the assumptions
of Theorem~1.  Theorem~2 is, by definition, restricted to the
semi-irrational case.
\end{theorem}

The weakly mixing requirement is essential for the last conclusion.  A
fixed nonlinear conjugate of a skew product may cease to look fibred in
the original coordinates, but it still has a nonlinear circle factor.
The constructions below instead eliminate all such factors.

\section{Two factor obstructions}

We isolate the two consequences of weak mixing which will be used in
both constructions.

\begin{lemma}[No skew-product conjugacy]\label{lem:no-skew}
Let $f$ preserve a full-support probability measure $\mu$ on $\T^2$ and
suppose that $(f,\mu)$ is weakly mixing.  Then $f$ is not topologically
conjugate to a skew product over a circle rotation.
\end{lemma}

\begin{proof}
Suppose that a homeomorphism $K$ satisfies
\[
 KfK^{-1}(x,y)=(x+\gamma,\Phi_x(y)).
\]
Then
\[
 u(z)=\exp\bigl(2\pi i\,\pi_1(Kz)\bigr)
\]
is a nonconstant continuous function and
\[
 u(fz)=\e^{2\pi i\gamma}u(z).
\]
Since $\mu$ has full support, $u$ is nonconstant in $L^2(\mu)$.  It is
therefore a nontrivial eigenfunction of the Koopman operator of $f$,
contrary to weak mixing.  Notice that nothing about the homotopy class,
smoothness, or measure behavior of $K$ was used.
\end{proof}

\begin{lemma}[Weak mixing forces unbounded mean motion]
\label{lem:wm-no-bmm}
Let $f\in\Homeo_0(\T^2)$ preserve a full-support probability measure and
be a pseudo-rotation.
\begin{enumerate}[label=\textup{(\alph*)}]
 \item If $\rho(F)$ is totally irrational and $f$ has bounded mean
       motion, then $f$ has a continuous factor onto the rigid
       translation $R_{\rho(F)}$.
 \item If $\rho(F)=(\alpha,0)$ with $\alpha\notin\Q$ and $f$ has bounded
       mean motion, then $f$ has a nontrivial measurable factor onto
       $R_\alpha$.
\end{enumerate}
Consequently, in either case weak mixing implies failure of bounded mean
motion.
\end{lemma}

\begin{proof}
Part~(a) is J\"ager's conservative linearization theorem
\cite[Theorem~C]{Jager09}.  In part~(b), bounded mean motion implies
bounded deviation in the primitive direction $e_1$.  The construction in
\cite[Proposition~2.1 and Remark~2.2]{Jager09} gives a measurable map
$\pi\colon\T^2\to\T$ satisfying
\[
 \pi\circ f=R_\alpha\circ\pi.
\]
Equivalently, on the universal cover it is obtained from the bounded
family of first-coordinate deviations, for example through a supremum
of the functions
$\pi_1(F^n(z))-n\alpha$.  In either case, pulling back a nontrivial
character of the rotation factor gives a nonconstant measurable
eigenfunction, which is incompatible with weak mixing.
\end{proof}

\section{Arithmetic and a lacunary cocycle}

The explicit special-flow construction uses an irrational number which
also meets the arithmetic assumptions of \cite{CWZ}.  Let
$Q_j=10^{m_j}$, where $m_1\geq2$ and
\begin{equation}\label{eq:Q-growth}
 Q_{j+1}\geq Q_j\exp(jQ_j^2),
 \qquad Q_j\mid Q_{j+1}.
\end{equation}
For $\boldsymbol\sigma=(\sigma_j)_{j\geq1}\in\{1,3\}^{\N}$, put
\begin{equation}\label{eq:alpha-family}
 \alpha=\alpha_{\boldsymbol\sigma}
 :=\sum_{j=1}^\infty\frac{\sigma_j}{Q_j},
 \qquad
 \alpha^{(j)}=\sum_{k=1}^j\frac{\sigma_k}{Q_k}
 =\frac{p_j}{Q_j}.
\end{equation}
We increase $Q_1$ so that $0<\alpha<\tfrac14$ and define
\begin{equation}\label{eq:theta}
 \theta_j=\|Q_j\alpha\|_{\T}
 =Q_j(\alpha-\alpha^{(j)}).
\end{equation}
Then
\begin{equation}\label{eq:theta-bounds}
 \frac{Q_j}{Q_{j+1}}\leq\theta_j
 \leq4\frac{Q_j}{Q_{j+1}},
 \qquad
 \frac{\theta_{j+1}}{\theta_j}\longrightarrow0.
\end{equation}

\begin{lemma}\label{lem:arithmetic}
Every $\alpha$ in \eqref{eq:alpha-family} satisfies
\begin{equation}\label{eq:SSL}
 \liminf_{n\to\infty}\frac1n\log\|n\alpha\|_{\T}=-\infty
\end{equation}
and is of strong non-Brjuno type.  It is transcendental,
$(\alpha^2,\alpha)$ is totally irrational, and
\begin{equation}\label{eq:vector-SSL}
 \liminf_{n\to\infty}\frac1n
 \log\|n(\alpha^2,\alpha)\|_{\T^2}=-\infty.
\end{equation}
\end{lemma}

\begin{proof}
Equations \eqref{eq:Q-growth}--\eqref{eq:theta-bounds} give
\[
 \frac1{Q_j}\log\|Q_j\alpha\|_{\T}
 =\frac1{Q_j}\log\theta_j\longrightarrow-\infty.
\]
Moreover, $p_j/Q_j$ is reduced and eventually satisfies Legendre's
criterion.  Hence $Q_j$ is a continued-fraction denominator.  If
$Q_j^+$ denotes the following denominator, then
$Q_j^+\geq(2\theta_j)^{-1}$, so
\[
 \frac{\log Q_j^+}{Q_j}\longrightarrow\infty.
\]
The nonnegative Brjuno series therefore diverges.  In particular,
$\alpha$ is Liouville and hence transcendental, so
$1,\alpha,\alpha^2$ are linearly independent over $\Q$.

Finally $Q_j\alpha=p_j+\theta_j$, and direct expansion gives
\[
 \|Q_j^2(\alpha^2,\alpha)\|_{\T^2}
 \leq C Q_j\theta_j.
\]
Dividing $\log(CQ_j\theta_j)$ by $Q_j^2$ and using
\eqref{eq:Q-growth} proves \eqref{eq:vector-SSL}.
\end{proof}

For $\boldsymbol t=(t_j)_{j\geq1}\in[1,2]^{\N}$ define
\begin{equation}\label{eq:h}
 h(x)=\sum_{j=1}^\infty
 b_j\cos(2\pi Q_jx),
 \qquad
 b_j=2^{-j}t_j\theta_j^{1-\delta}.
\end{equation}
Write $S_nh(x)=\sum_{\ell=0}^{n-1}h(x+\ell\alpha)$.

\begin{proposition}\label{prop:Birkhoff}
The function $h$ belongs to $C^\infty(\T)$, has zero integral, and
there is $C_\delta<\infty$ such that
\begin{equation}\label{eq:Birkhoff-upper}
 \sup_x|S_nh(x)|\leq C_\delta n^\delta
 \qquad(n\geq1).
\end{equation}
There are $N_j\to\infty$ and $x_j\in\T$ such that
\begin{equation}\label{eq:Birkhoff-lower}
 S_{N_j}h(x_j)\geq
 c_\delta2^{-j}\theta_j^{-\delta}\longrightarrow\infty.
\end{equation}
For every $0\leq\delta'<\delta$,
\begin{equation}\label{eq:Birkhoff-sharp}
 N_j^{-\delta'}S_{N_j}h(x_j)\longrightarrow\infty.
\end{equation}
The constants are uniform in $\boldsymbol t\in[1,2]^{\N}$.
\end{proposition}

\begin{proof}
For a fixed derivative order $r$, the $C^r$ size of the $j$th mode is
at most a constant times
\[
 2^{-j}Q_j^r\theta_j^{1-\delta},
\]
which is summable by \eqref{eq:Q-growth}.  Thus $h$ is smooth and its
mean is zero.  The geometric-sum bound
\begin{equation}\label{eq:geometric}
 \left|\sum_{\ell=0}^{n-1}
 \e^{2\pi iQ_j(x+\ell\alpha)}\right|
 \leq C_0\min\{n,\theta_j^{-1}\}
\end{equation}
gives
\[
 |S_nh(x)|\leq C_1\sum_{j\geq1}
 2^{-j}\theta_j^{1-\delta}\min\{n,\theta_j^{-1}\}.
\]
Choose $J$ with
$\theta_J^{-1}\leq n<\theta_{J+1}^{-1}$.  Strong lacunarity gives
\[
 \sum_{j\leq J}2^{-j}\theta_j^{-\delta}\leq C_2n^\delta,
 \qquad
 n\sum_{j>J}2^{-j}\theta_j^{1-\delta}\leq C_3n^\delta,
\]
which proves \eqref{eq:Birkhoff-upper}.

Set $N_j=\lfloor(4\theta_j)^{-1}\rfloor$.  Choose $x_j$ so that the
$Q_j$th geometric sum has positive maximal phase.  Its contribution is
at least $c_0 2^{-j}\theta_j^{-\delta}$.  The contributions of the
earlier modes and the later modes are respectively bounded by
\[
 C\sum_{k<j}2^{-k}\theta_k^{-\delta}
 \quad\text{and}\quad
 CN_j\sum_{k>j}2^{-k}\theta_k^{1-\delta}.
\]
Both are $o(2^{-j}\theta_j^{-\delta})$.  This proves
\eqref{eq:Birkhoff-lower}.  Since $N_j\asymp\theta_j^{-1}$,
\eqref{eq:Birkhoff-sharp} follows as well.
\end{proof}

\section{The controlled Anosov--Katok construction}

Let $R_v(z)=z+v$ on $\T^2$.  A rational $v$ gives a periodic
translation.  Let $(\mathcal P_j)$ be increasing finite algebras of
rational rectangles which generate the Borel sigma-algebra, and let
$\varepsilon_j\downarrow0$.

\begin{proposition}[Weakly mixing A--K step]\label{prop:AK-step}
Suppose $v_j\in\Q^2/\Z^2$, $H_j\in\Diff_0^\infty(\T^2,\Leb)$, and
\[
 f_j=H_jR_{v_j}H_j^{-1}.
\]
Given a $C^\infty$ neighborhood $\mathcal U$ of $f_j$, a neighborhood
$V$ of $v_j$, and a finite algebra $\mathcal P$, one can find
\begin{enumerate}[label=\textup{(\roman*)}]
 \item an area-preserving $h_{j+1}\in\Diff_0^\infty(\T^2)$ commuting
       with $R_{v_j}$,
 \item a rational vector $v_{j+1}\in V$ of arbitrarily large period,
       and a time $m_{j+1}$,
\end{enumerate}
such that, with $H_{j+1}=H_jh_{j+1}$,
\[
 H_{j+1}R_{v_{j+1}}H_{j+1}^{-1}\in\mathcal U
\]
and
\begin{equation}\label{eq:finite-mixing-test}
 \left|\Leb(A\cap f_{j+1}^{-m_{j+1}}B)
       -\Leb(A)\Leb(B)\right|<\varepsilon
 \qquad(A,B\in\mathcal P).
\end{equation}
The period and the closeness of $v_{j+1}$ may be required to satisfy
any additional finite list of lower and upper bounds.
\end{proposition}

\begin{proof}[Source and construction principle]
This is the toral-translation form of the weakly mixing density step in
the Anosov--Katok method \cite[Section~5]{AK70}; see also
\cite[Sections~2.1--2.4]{FK04}.  A fundamental domain for the finite
cyclic action generated by $R_{v_j}$ is divided into small rectangles.
An area-preserving diffeomorphism, equal to the identity near their
boundaries, redistributes most of these rectangles.  It is extended to
the other translates by the commutation rule.  A sufficiently long
nearby rational orbit then visits the redistributed rectangles with the
frequencies required in \eqref{eq:finite-mixing-test}.  The
redistribution is chosen first and the new rational parameter second;
hence the latter can be taken arbitrarily close to $v_j$.  This last
freedom simultaneously gives $C^\infty$ convergence and any prescribed
finite-time shadowing estimates.  The construction can be performed
with conjugacies isotopic to the identity.  For a fixed circle-action
parameter, the corresponding quantitative weakly mixing construction is
carried out in \cite{FS05}.
\end{proof}

We now insert the quantitative deviation requirement.  The argument is
independent of whether the limiting vector is semi-irrational or totally
irrational.

\begin{proposition}[Quantitative slowdown]\label{prop:AK-slowdown}
For every $0<\delta<1$ the weakly mixing A--K construction can be arranged
so that its limit $f$ has a lift $F$ and a vector $\omega$ satisfying
\begin{equation}\label{eq:AK-Cdelta}
 \sup_z|F^n(z)-z-n\omega|\leq Cn^\delta
 \qquad(n\geq1).
\end{equation}
The limiting vector may be chosen either as $(\alpha,0)$ with $\alpha$
strongly super-Liouville and of strong non-Brjuno type, or as a totally
irrational strongly super-Liouville vector.
\end{proposition}

\begin{proof}
Run Proposition~\ref{prop:AK-step} with
$\mathcal P=\mathcal P_j$ and $\varepsilon=\varepsilon_j$, taking the
neighborhoods successively smaller.  Let $\omega=\lim_jv_j$ and put
\[
 \widehat f_j=H_jR_\omega H_j^{-1}.
\]
Choose lifts
\[
 \widetilde H_j(z)=z+\eta_j(z),
\]
where $\eta_j$ is $\Z^2$-periodic, and set
\[
 B_j=2\|\eta_j\|_\infty.
\]
For every $m\in\Z$ and $w=\widetilde H_j^{-1}z$,
\begin{equation}\label{eq:AK-stage-bound}
 \widehat F_j^m(z)-z-m\omega
 =\eta_j(w+m\omega)-\eta_j(w),
\end{equation}
so the left-hand side has norm at most $B_j$.

After $H_j$ has been chosen, choose an integer $L_j>L_{j-1}$ such that
\begin{equation}\label{eq:Lj}
 B_j+1\leq\tfrac14L_j^\delta.
\end{equation}
At the next parameter selections require
\begin{equation}\label{eq:AK-shadow}
 \max_{1\leq m\leq L_{j+1}}
 \|F^m-\widehat F_j^m\|_{C^0}\leq1.
\end{equation}
This is a legitimate look-ahead condition.  Indeed, for a fixed
conjugacy $H$,
\[
 d_{C^0}\bigl((HR_uH^{-1})^m,(HR_vH^{-1})^m\bigr)
 \leq C(H)m|u-v|.
\]
The new conjugacy commutes with the old rational translation.  Hence
the difference between two consecutive $\widehat f_j$'s has the same
form, with a finite constant depending on the already chosen
conjugacies.  Once $L_{j+1}$ is known, Proposition~\ref{prop:AK-step}
allows $v_{j+2}$ and all later parameters to be chosen so close that the
sum of the finite-time errors is at most one.  The same choices may be
made small in every $C^r$ norm, so $f_j$ and $\widehat f_j$ have the
same $C^\infty$ limit $f$.

For $L_j\leq m<L_{j+1}$, equations
\eqref{eq:AK-stage-bound}--\eqref{eq:AK-shadow} give
\[
 \sup_z|F^m(z)-z-m\omega|
 \leq1+B_j\leq\tfrac14L_j^\delta\leq\tfrac14m^\delta.
\]
Changing the constant handles the finitely many initial times and
proves \eqref{eq:AK-Cdelta}.  Uniform sublinear convergence also proves
that the rotation set is $\{\omega\}$.

For the semi-irrational construction, keep $v_j=(\alpha_j,0)$ and make
the successive denominators grow so rapidly that the limit $\alpha$
satisfies \eqref{eq:SSL} and the strong non-Brjuno condition.  This is
the usual circle-action version of the construction \cite{FS05}.

For the totally irrational construction, choose $v_j\in\Q^2$ and
nested rational balls whose diameters are smaller than
$\exp(-j q_j^2)/q_j$, where $q_jv_j\in\Z^2$.  At stage $j$ keep the
next ball away from the first $j$ rational resonance hyperplanes
\[
 \{v:\langle k,v\rangle\in\Z\},
 \qquad 0<|k|\leq j.
\]
The intersection vector $\omega$ is totally irrational and
\[
 \|q_j\omega\|_{\T^2}\leq\exp(-jq_j^2),
\]
which is more than the strong super-Liouville condition required in
\cite{CWZ}.  These arithmetic restrictions only shrink the parameter
neighborhood $V$ and are compatible with Proposition~\ref{prop:AK-step}.
\end{proof}

\begin{proof}[Completion of the A--K part of Theorem~\ref{thm:main}]
The finite correlation tests \eqref{eq:finite-mixing-test}, applied to
the generating algebras with errors tending to zero and preserved by
the future perturbations, imply weak mixing of $f$.  Proposition
\ref{prop:AK-slowdown} gives the singleton rotation set and the
$(C,\delta)$ estimate.  Lemma~\ref{lem:wm-no-bmm} gives failure of
bounded mean motion in both rotation types, and
Lemma~\ref{lem:no-skew} excludes conjugacy to every skew product over a
circle rotation.
\end{proof}

\section{The explicit special-flow construction}

We now use the same Fourier series \eqref{eq:h} to obtain a completely
explicit weakly mixing input.  The required spectral criterion is the
following standard consequence of the Katok criterion for special
flows; see \cite[Theorem~5.60]{KT06}.

\begin{proposition}[Fourier weakly mixing criterion]\label{prop:KT}
Let $g(x)=\sum_{n\neq0}\widehat g(n)\e^{2\pi inx}$ be real-valued,
$C^2$, and centered.  Suppose that for rational approximations
$p_j/q_j$ of $\alpha$,
\begin{equation}\label{eq:KT-criterion}
 \frac{q_j|\alpha-p_j/q_j|}
 {\sum_{k\geq1}|\widehat g(kq_j)|}\longrightarrow0,
 \qquad
 \frac{|\widehat g(q_j)|}
 {\sum_{k\geq1}|\widehat g(kq_j)|}\geq c>0.
\end{equation}
Then every positive roof $g_0+\varepsilon g$ gives a weakly mixing
special flow over $R_\alpha$.
\end{proposition}

\begin{lemma}\label{lem:h-weak-mixing}
For every sufficiently small $\varepsilon>0$, the special flow over
$R_\alpha$ under
\begin{equation}\label{eq:roof}
 r(x)=1+\varepsilon h(x)
\end{equation}
is weakly mixing.
\end{lemma}

\begin{proof}
Use $q_j=Q_j$.  Since $Q_j\mid Q_\ell$ for $\ell>j$, the positive
Fourier coefficients at multiples of $Q_j$ are exactly the $j$th mode
and a subfamily of the later modes.  Strong lacunarity gives
\[
 \sum_{k\geq1}|\widehat h(kQ_j)|
 =\frac{b_j}{2}(1+o(1)),
 \qquad
 \frac{|\widehat h(Q_j)|}
 {\sum_{k\geq1}|\widehat h(kQ_j)|}\longrightarrow1.
\]
Furthermore,
\[
 Q_j\left|\alpha-\frac{p_j}{Q_j}\right|=\theta_j,
\]
and hence
\[
 \frac{Q_j|\alpha-p_j/Q_j|}
 {\sum_{k\geq1}|\widehat h(kQ_j)|}
 \leq C2^j\theta_j^\delta\longrightarrow0.
\]
Proposition~\ref{prop:KT} applies.  Positivity follows by reducing
$\varepsilon$ if necessary.
\end{proof}

We realize this special flow smoothly on a torus.  Consider
\begin{equation}\label{eq:mapping-torus}
 M_\alpha=(\T\times\R)/((x,u+1)\sim(x+\alpha,u))
\end{equation}
and identify it with $\T^2$ by
\begin{equation}\label{eq:Lalpha}
 L_\alpha[x,u]=(x+\alpha u,u)\pmod{\Z^2}.
\end{equation}
Let $X_0$ generate vertical translation on $M_\alpha$.  Choose
$\chi\in C_c^\infty((0,1))$ with $\chi\geq0$ and
$\int_0^1\chi(u)\dd u=1$.  Put
\begin{equation}\label{eq:w}
 w([x,u])=1+\varepsilon\chi(u)h(x)>0,
 \qquad Y=w^{-1}X_0,
\end{equation}
and let $\psi^t$ be the flow of $Y$.  The first return map to $u=0$ is
$R_\alpha$ and its return time is \eqref{eq:roof}.  The smooth area form
\begin{equation}\label{eq:invariant-area}
 \Omega=w\,\dd x\wedge\dd u
\end{equation}
is invariant and has total mass one.

Let $p=[x,u]$, $0\leq u<1$.  Suppose the orbit up to time $t$ makes
$N=N_t(p)$ crossings and ends at $[x+N\alpha,u_t]$.  In the lift of
the coordinates \eqref{eq:Lalpha},
\begin{equation}\label{eq:SF-displacement}
 \widetilde\psi^t(z)-z=(N+u_t-u)(\alpha,1).
\end{equation}
The changed clock gives
\begin{equation}\label{eq:clock}
 t=N+u_t-u+\varepsilon A_t(p),
\end{equation}
where
\begin{equation}\label{eq:endpoint-error}
 |A_t(p)-S_Nh(x)|\leq2\|h\|_\infty.
\end{equation}
Consequently,
\begin{equation}\label{eq:SF-exact}
 \widetilde\psi^t(z)-z-t(\alpha,1)
 =-\varepsilon A_t(p)(\alpha,1).
\end{equation}
Since $N\leq C(1+t)$, Proposition~\ref{prop:Birkhoff} yields
\begin{equation}\label{eq:SF-upper}
 \sup_z|\widetilde\psi^t(z)-z-t(\alpha,1)|
 \leq C(1+t^\delta).
\end{equation}

Fix $\tau>0$ and let $g_\tau=\psi^\tau$.  Equation
\eqref{eq:SF-upper} at $t=n\tau$ proves the $(C,\delta)$ estimate and
shows that the rotation set is
\begin{equation}\label{eq:SF-rotation-natural}
 \rho(\widetilde g_\tau)=\{\tau(\alpha,1)\}.
\end{equation}
At the return times
\[
 t_j=S_{N_j}r(x_j)=N_j+\varepsilon S_{N_j}h(x_j),
\]
equation \eqref{eq:SF-exact} gives a deviation of size comparable to
$S_{N_j}h(x_j)$.  Choose $k_j\in\Z$ with
$|k_j\tau-t_j|\leq\tau/2$.  The displacement during a bounded flow-time
interval is uniformly bounded.  Therefore the deviation of
$g_\tau^{k_j}$ differs by a bounded amount from the deviation at $t_j$.
Equations \eqref{eq:Birkhoff-lower}--\eqref{eq:Birkhoff-sharp} show
that $g_\tau$ has unbounded mean motion and satisfies no smaller
deviation exponent.

The flow $\psi$ is weakly mixing by Lemma~\ref{lem:h-weak-mixing}.
Every nonzero time map of a weakly mixing flow is weakly mixing: an
eigenfunction for one time map would have spectral measure supported on
a discrete arithmetic set and would yield an eigenfunction for the
flow.  Hence every $g_\tau$, $\tau\neq0$, is weakly mixing.

For $\tau=1$, replacing the natural lift by its translate by $(0,-1)$
gives
\begin{equation}\label{eq:SF-semi}
 \rho(\widetilde g_1-(0,1))=\{(\alpha,0)\}.
\end{equation}
For $\tau=\alpha$,
\begin{equation}\label{eq:SF-total}
 \rho(\widetilde g_\alpha)=\{(\alpha^2,\alpha)\},
\end{equation}
which is totally irrational and strongly super-Liouville by
Lemma~\ref{lem:arithmetic}.

Finally, Moser's theorem \cite{Moser65} gives
$J\in\Diff_0^\infty(M_\alpha)$ with
$J_*\Omega=\dd x\wedge\dd u$.  Define
\begin{equation}\label{eq:SF-final}
 f_{\mathrm{SF},\tau}
 =L_\alpha Jg_\tau J^{-1}L_\alpha^{-1},
 \qquad \tau\in\{1,\alpha\}.
\end{equation}
Then $f_{\mathrm{SF},\tau}$ preserves Lebesgue area.  An
identity-isotopic conjugacy changes a lifted displacement cocycle by the
difference of two values of a bounded periodic function.  Thus it
preserves the singleton rotation set, the $(C,\delta)$ estimate, failure
of bounded mean motion, and sharpness of the exponent.  It also
preserves weak mixing.  Lemma~\ref{lem:no-skew} now excludes every
linear or nonlinear conjugacy to a skew product.  This completes the
special-flow part of Theorem~\ref{thm:main}.

\section{Compatibility with the rigidity hypotheses}

The exponent range $0<\delta<\tfrac12$ is the range used in
Theorems~1 and~2 of \cite{CWZ}.  All maps constructed above are smooth,
hence Lipschitz and $C^k$ for every finite $k$, and preserve Lebesgue
area.

In the semi-irrational case, the representative $(\alpha,0)$ is
irrational and satisfies the strong super-Liouville condition by
\eqref{eq:SSL}.  The denominator growth also gives the strong
non-Brjuno condition.  Thus these examples meet the assumptions of
Theorem~1 with H\"older exponent one and of Theorem~2 for every finite
$k\geq2$.

In the totally irrational case, the A--K parameter is chosen strongly
super-Liouville in Proposition~\ref{prop:AK-slowdown}, while the
special-flow vector $(\alpha^2,\alpha)$ has this property by
\eqref{eq:vector-SSL}.  These examples therefore meet the assumptions
of Theorem~1.  They are outside the statement of Theorem~2 solely
because that theorem assumes a semi-irrational rotation vector.

The examples demonstrate that the $(C,\delta)$ hypothesis in those
results does not conceal bounded mean motion.  They also demonstrate
that the phenomenon is not caused by an underlying skew-product
factor: weak mixing rules out such a factor even after an arbitrary
nonlinear coordinate change.

\section{Abundance}

Let $\cE_{\delta,\mathrm{AK}}^{\mathrm{SI}}$ and
$\cE_{\delta,\mathrm{AK}}^{\mathrm{TI}}$ denote the semi-irrational and
totally irrational A--K examples above, and define
$\cE_{\delta,\mathrm{SF}}^{\mathrm{SI}}$ and
$\cE_{\delta,\mathrm{SF}}^{\mathrm{TI}}$ analogously.

\begin{theorem}\label{thm:abundance}
For every $0<\delta<\tfrac12$ and each
$\mathrm X\in\{\mathrm{AK},\mathrm{SF}\}$,
\[
 |\cE_{\delta,\mathrm X}^{\mathrm{SI}}|
 =|\cE_{\delta,\mathrm X}^{\mathrm{TI}}|
 =2^{\aleph_0}.
\]
Each of the four families contains continuum many topological conjugacy
classes, all of which contain no skew product over a circle rotation.
\end{theorem}

\begin{proof}
For the special-flow construction there are continuum many choices of
$\boldsymbol\sigma\in\{1,3\}^{\N}$ and, for each $\alpha$, continuum
many choices of $\boldsymbol t\in[1,2]^{\N}$.  The estimates and the
weakly mixing criterion are uniform over these choices.  For the A--K
construction, at every sufficiently late step Proposition
\ref{prop:AK-step} leaves at least two disjoint parameter neighborhoods
available.  Infinite binary branching gives continuum many limiting
parameters and maps while preserving all finite requirements.

If two homeomorphisms in $\Homeo_0(\T^2)$ are topologically conjugate,
their singleton rotation sets are related by the matrix in
$\mathrm{GL}(2,\Z)$ induced by the conjugacy.  Every
$\mathrm{GL}(2,\Z)$ orbit is countable.  Selecting one parameter from
each orbit therefore gives continuum many topological conjugacy classes
in the totally irrational families.  In the semi-irrational family,
if $A(\alpha,0)=(\alpha',0)$ modulo $\Z^2$ with
$0<\alpha,\alpha'<\tfrac14$, irrationality forces the relevant diagonal
entry of $A$ to be $\pm1$, and the chosen interval then forces
$\alpha'=\alpha$.  Thus distinct parameters already give continuum
many conjugacy classes.

There are at most continuum many continuous self-maps of the compact
metric space $\T^2$, giving the matching upper bound.  The final
statement follows from weak mixing and Lemma~\ref{lem:no-skew}.
\end{proof}

\section{Conclusion}

For every $0<\delta<\tfrac12$, the $(C,\delta)$-deviation condition is
strictly weaker than bounded mean motion even after one excludes all
systems carrying a hidden skew-product structure.  The controlled
Anosov--Katok construction gives weakly mixing semi-irrational and
totally irrational pseudo-rotations with a prescribed sublinear
all-time deviation scale.  The explicit special-flow construction goes
further: one lacunary smooth roof simultaneously satisfies the Fourier
criterion for weak mixing, the uniform $O(n^\delta)$ Birkhoff estimate,
and a sharp unbounded lower estimate.  After area normalization, both
methods yield smooth Lebesgue-preserving pseudo-rotations satisfying the
arithmetic hypotheses of the relevant rigidity theorems in \cite{CWZ},
but none is conjugate to a skew product by any linear or nonlinear
coordinate change.

\end{document}